\documentclass[11pt,reqno]{amsart}
    \usepackage{amssymb}
    \usepackage{amsmath}
    \usepackage{mathrsfs}
    \usepackage{mathtools}
    \usepackage{amssymb, amsmath, amsfonts,tikz}
    \usepackage[title]{appendix}
    \usepackage{mathabx}
   \usepackage{color}
\usetikzlibrary{patterns}
    \usepackage[hidelinks]{hyperref}
    \allowdisplaybreaks
    \usepackage[numbers,sort&compress]{natbib}

    \makeatletter
    \def\tank#1{\mathbb Protected@xdef\@thanks{\@thanks
     \mathbb Protect\footnotetext[0]{#1}}}
    \def\bigfoot{

     \@footnotetext}
    \makeatother

    \newcommand{\ea}{\end{array}}

    \allowdisplaybreaks
    \numberwithin{equation}{section}

    \allowdisplaybreaks

    \newtheorem{theorem}{Theorem}[section]
    \newtheorem{lemma}{Lemma}[section]
    \newtheorem{proposition}[theorem]{Proposition}

    \newtheorem{corollary}[theorem]{Corollary}

    \def\beq{\begin{equation}}
    \def\nneq{\end{equation}}

    \def\bthm{\begin{theorem}}
    \def\nthm{\end{theorem}}

    \def\blem{\begin{lemma}}
    \def\nlem{\end{lemma}}
    \def\bprf{\begin{proof}}
    \def\nprf{\end{proof}}
    \def\bprop{\begin{prop}}
    \def\nprop{\end{prop}}
    \def\brmk{\begin{rem}}
    \def\nrmk{\end{rem}}

    \def\bexa{\begin{exa}}
    \def\nexa{\end{exa}}
    \def\bcor{\begin{cor}}
    \def\ncor{\end{cor}}

      \def\R{\mathbb{R}}

    \newcommand{\1}{{\bf 1}}

\newcommand*\DAlambert{\mathop{} \mathbin\Box}

    \title[Local linearization  for the damped stochastic Klein-Gordon equation]{Local linearization for  the   nonlinear  damped stochastic Klein-Gordon equation}

         \date{}
    
\begin{document}

    \author[G. Rang]{Guanglin Rang}
    \address[]{Guanglin Rang, School of Mathematics and Statistics,  Wuhan University,  Wuhan, 430072,
    China.}
    \email{glrang.math@whu.edu.cn}

    \author[R. Wang]{Ran Wang}
    \address[]{Ran Wang, School of Mathematics and Statistics,  Wuhan University,  Wuhan, 430072,
    China.}
    \email{rwang@whu.edu.cn}

    \maketitle
     \noindent {\bf Abstract:}  For the $1+1$ dimensional nonlinear damped stochastic Klein-Gordon equation driven by space-time white noise, we  prove that the second-order increments of the solution can be approximated,  after scaling with the diffusion coefficient,  by those of the corresponding linearized stochastic Klein-Gordon equation.    This extends the result of Huang et al. \cite{HOO2024} for the stochastic wave equation.  A key difficulty arises from the more  complex structure of the Green function, which we overcome by means of subtle analytical estimates.     As applications, we analyze the quadratic variation of the solution and construct a consistent estimator for the diffusion parameter.
 \vskip0.3cm

    \section{Introduction}
Consider the nonlinear    damped stochastic Klein-Gordon equation:
      \begin{equation}\label{eq:DSKG}
\begin{cases}
\vspace{6pt}
\displaystyle{\left (\DAlambert+a\partial_t+m^2 \right)u(t,x)= F(u(t,x))\dot    W(t,x), \quad t>0, x \in \mathbb R,}\\
\displaystyle{u(0, x) = 0, \quad \frac{\partial}{\partial t} u(0, x) = 0.}
\end{cases}
\end{equation}
Here,  $\DAlambert$  is the d'Alembertian in one-dimensional space  defined by 
\begin{align}\label{eq dAl}
\DAlambert:=\partial_{tt}-\partial_{xx},
\end{align} 
the constants $a \in \mathbb R$ and $m\in \mathbb R_+$ are known from the physics literature as the damping constant and mass, the diffusion coefficient   $F:\mathbb R\rightarrow\mathbb R$  is a   globally  Lipschitz continuous function, and $\dot W$ is a Gaussian space-time white noise defined on a complete probability space $(\Omega, \mathcal {F}, \mathbb P)$.   The case $a>0$ corresponds to ``damping", whereas    $a<0$  corresponds to ``excitation".  
   When  $a=0, m=0$, \eqref{eq:DSKG} reduces to the well-known stochastic wave equation, see, e.g.,    \cite{D99, DS15, Walsh86}.

 It is well known that   stochastic partial differential  equations with multiplicative noise has  local linearization behavior; see, e.g.,  \cite{FKM2015, HP15}.   To state this property precisely, we first introduce the following stochastic heat equation   (SHE, for short):
        \begin{equation}\label{SHE}
\begin{cases}
\vspace{6pt}
\displaystyle{\frac{\partial}{\partial t} X(t, x) =\frac{\partial^2}{\partial x^2} X(t, x) +   F(X(t,x)) \dot{W}(t, x), \quad t \ge 0, x \in \mathbb{R},}\\
\displaystyle{X(0, x) = 0,}
\end{cases}
\end{equation} 
where  $\dot W$ and  $F$ are the same as in \eqref{SHE}.  Let   $Z$ denote the linearized version of $X$, i.e.,    the solution to the   SHE \eqref{SHE} with $F\equiv1$. Then,  the spatial increments of $X$ satisfy 
  \begin{equation}\label{SHE appro}
  X(t, x + h) - X(t, x) = F(X(t, x)) \left\{Z(t, x + h) - Z(t, x)\right\} +  \mathcal R_{t, x}(h),
\end{equation}  
where  the remainder term $\mathcal R_{t, x}(h)$ converges to zero   faster than $Z(t, x + h) -Z(t, x)$   as $h\to 0$.   Relation \eqref{SHE appro} shows that, for each fixed $t>0$, the  local spatial  fluctuations  of    $X(t, \cdot)$ are essentially governed by those of  $Z(t, \cdot)$.   In other words,  ignoring precise regularity conditions,   \eqref{SHE appro} indicates  that $X(t, \cdot)$ is controlled by $Z(t, \cdot)$ in the sense of Gubinelli's theory of  controlled paths     \cite{G04}.  
  
  An analogous local linearization holds   for  temporal increments: for fixed $t>0$ and $x\in \mathbb R$, as $\varepsilon \downarrow 0$, the increment  $X(t+\varepsilon, x)-X(t, x)$ admits a similar structure; see \cite{KSXZ2013, HP15, WX2024, QWWX2025} and references therein. Those results provide a   quantitative framework for analyzing the local   structure   of   sample paths  of the solution, including properties such as  Khinchin's law of the iterated logarithm,  Chung's  law of the iterated logarithm,    quadratic variation, and  small-ball probability estimates; see, e.g.,  \cite{CHKK19, DNP2025, Das2024,  HK2017,  HL2025, KKM23, KLPX25} and references therein.   
 
 In sharp contrast to the parabolic setting,  the stochastic wave equation does not admit local linearization from a one-parameter perspective. This key feature was established by Huang and Khoshnevisan, who derived the asymptotic behavior of   the temporal and spatial quadratic variations individually through a sequence of approximations in \cite{HK2016}. Subsequently, Tudor and Zurcher provided an alternative proof using Malliavin calculus and further extended the discussion to applications in statistical inference in \cite{TZ2025}.

  Recently,  from a bi-parameter perspective,  Huang  et al.   \cite{HOO2024}  revisited the local linearization problem for the stochastic wave equation and showed that such linearization becomes possible when both temporal and spatial parameters are considered jointly.       
  
  In this work, we extend   the bi-parameter local linearization result of \cite{HOO2024} to the nonlinear damped stochastic Klein-Gordon equation \eqref{eq:DSKG}.    We show that the linearization framework is applicable to the equation \eqref{eq:DSKG}   when damping or mass terms are included.

       To carry out this analysis,  let us first give the    linear damped stochastic Klein-Gordon in one spatial dimension: 
      \begin{equation}\label{eq:DSKG linear}
\begin{cases}
\vspace{6pt}
\displaystyle{\left(\DAlambert+a\partial_t+m^2\right )U(t,x)=\dot
    W(t,x), \quad t>0, x \in \R,}\\
\displaystyle{U(0, x) = 0, \quad \frac{\partial}{\partial t} U(0, x) = 0.}
\end{cases}
\end{equation}

 Consider a   coordinate system $(\tau, \lambda)$ obtained by rotating the $(t, x)$-coordinates by $-45^\circ$.  That is, 
\begin{equation}\label{eq transform}
(\tau, \lambda): = \Big(\frac{t-x}{\sqrt{2}}, \frac{t+x}{\sqrt{2}}\Big) \quad \text{and} \quad (t, x) 
= \Big( \frac{\tau+\lambda}{\sqrt{2}}, 
\frac{-\tau+\lambda}{\sqrt{2}} \Big). 
\end{equation}
The  axes $\tau$ and $\lambda$ correspond to  the characteristic line directions of the wave equation, see \cite{QT07,Walsh86}.                  
        For $\tau>0$ and $\lambda\ge -\tau$,  define 
\begin{equation}\label{SWE 2}
v(\tau, \lambda) := u\left(\frac{\tau+\lambda}{\sqrt{2}}, \frac{-\tau+\lambda}{\sqrt{2}}\right),
\end{equation} 
and
\begin{equation}\label{SWE 3}
V(\tau, \lambda) := U\left(\frac{\tau+\lambda}{\sqrt{2}}, \frac{-\tau+\lambda}{\sqrt{2}}\right). 
\end{equation}

 For  any $\varepsilon > 0$, define the difference operators $\delta_{\varepsilon}^{(j)}, j=1, 2$  by  
\begin{equation}\label{eq diff}
 \begin{split}
 \delta_{\varepsilon}^{(1)}f(\tau, \lambda):=&\, f(\tau+\varepsilon, \lambda)-f(\tau, \lambda),\\
   \delta_{\varepsilon}^{(2)}f(\tau, \lambda):=&\, f(\tau, \lambda+\varepsilon)-f(\tau, \lambda).
 \end{split}
 \end{equation} 
Define the remainder terms $R_{\varepsilon}^+(\tau, \lambda)$ and $R_{\varepsilon}^-(\tau, \lambda)$ as follows:  
 \begin{equation}\label{eq R}
\begin{split}
 R_{\varepsilon}^{\pm}(\tau, \lambda) := & \,\delta_{\pm\varepsilon}^{(1)}\delta_{\varepsilon}^{(2)}v(\tau, \lambda)-F(v(\tau, \lambda))\delta_{\pm\varepsilon}^{(1)}\delta_{\varepsilon}^{(2)}V(\tau, \lambda)\\
=&\, \left\{ v(\tau\pm\varepsilon, \lambda+\varepsilon)-v(\tau\pm\varepsilon, \lambda) - v(\tau, \lambda+\varepsilon)+v(\tau, \lambda)\right\}\\
&-F(v(\tau, \lambda))\left\{V(\tau\pm\varepsilon, \lambda+\varepsilon)-V(\tau\pm\varepsilon, \lambda)-V(\tau, \lambda+\varepsilon)+V(\tau, \lambda) \right\}.
\end{split}
\end{equation}

      \begin{theorem}\label{thm error 1}   For any $p \ge 1, M>0$ and $L>0$, there exists a constant $c(p, M, L)>0$ such that 
\begin{align}\label{eq error}
\left\| R_{\varepsilon}^{\pm}(\tau, \lambda)\right\|_{L^p(\Omega)} \le c(p, M, L)\varepsilon^{ \frac{3}{2}},
\end{align}
holds uniformly for all $\tau \in [0, M]$, $\lambda \in [-\tau, L]$,  and  all sufficiently small $\varepsilon > 0$.
\end{theorem}

  For a   function $f(t, x)$, define    the difference operators $\Delta_{\varepsilon}^{(1)}$ and $\Delta_{\varepsilon}^{(2)}$  as follows:
\begin{equation}\label{eq Delta1-2}
\begin{split}
\Delta_{\varepsilon}^{(1)}f(t,x)=&\, f(t,x+2\varepsilon)-f(t-\varepsilon, x+\varepsilon)-f(t+\varepsilon, x+\varepsilon)+f(t,x),\\
\Delta_{\varepsilon}^{(2)}f(t,x)=&\, f(t+2\varepsilon, x)-f(t+\varepsilon, x-\varepsilon)-f(t+\varepsilon, x+\varepsilon)+f(t,x).
\end{split}
\end{equation} 
  
  As an immediate consequence of Theorem \ref{thm error 1}, we obtain the following bi-parameter local linearization result for the solution of  the equation \eqref{eq:DSKG} in the original coordinates.
  \begin{proposition}\label{thm error 2} Let $u$ and $U$ be the solutions to equations \eqref{eq:DSKG} and \eqref{eq:DSKG linear}, respectively. Then,  for any $p \ge 1$, $M>0$, and $L>0$,  the following estimate holds   uniformly for $t\in [0, M]$, $x\in [-L, L]$,   all sufficiently small $\varepsilon > 0$, and $k=1, 2$:  
\begin{equation}\label{eq error 2}
\begin{split}
 \left\| \Delta_{\varepsilon}^{(k)} u(t, x) - F(u(t,x)) \Delta_{\varepsilon}^{(k)} U(t, x) \right\|_{L^p(\Omega)}  \le c(p, M, L)   \varepsilon^{  \frac{3}{2}}.
\end{split}
\end{equation} 
\end{proposition}

       For each integer $N\ge1$, define the grid points  
     \begin{align}\label{eq grid}
      \tau_i:=\frac{i}{N},  \ \ \ \text{and } \lambda_j:=\frac{j}{N} \ \ \ \text{for } i, j=0, \cdots, N.
       \end{align}  
      The   quadratic  variation of   the process $v$ given in \eqref{SWE 2}  is defined as 
  \begin{align}\label{eq V v}
Q_{N}(v):=  \sum_{i=0}^{N-1}\sum_{j=0}^{N-1}\Delta_{i,j}(v)^2,
\end{align} 
with
\begin{align}\label{eq V ij}
\Delta_{i,j}(v)  := v\left(\tau_{i+1}, \lambda_{j+1}  \right)-v\left(\tau_{i+1}, \lambda_{j}\right)-v\left(\tau_{i}, \lambda_{j+1}  \right) +v\left(\tau_{i}, \lambda_{j}    \right).
\end{align}

 We now investigate  the   asymptotic behavior of the quadratic variation sequence $\{Q_{N}(v)\}_{ N\ge 1}$  as $N\rightarrow\infty$.  
The core of our approach is to  approximate each increment $\Delta_{i,j}(v)$  by the linearized term $F(v(\tau_i, \lambda_j))\Delta_{i,j}(V)$, where $V$ denotes  the Gaussian field  defined in  \eqref{SWE 3}.    
    The  limit  characterized in the following  proposition  provides a direct  pathway  to  estimating the diffusion parameter. 
    \begin{proposition}\label{prop variation converge}  
  There exists a constant $c>0$ such that,  for every integer    $N\ge 1$, 
  \begin{align}\label{eq quad}
  \mathbb E\left[\left| Q_{N}(v)-  \frac{1}{4} \int_0^1\int_0^1   F^2\big(v(\tau, \lambda) \big) d\tau d\lambda \right|  \right]\le c N^{- \frac12}.
  \end{align} 
   \end{proposition}

The remainder of this paper is structured as follows.   In Section \ref{Sec 2}, we compile preliminary results on stochastic integration and  solution properties for  the equation \eqref{eq:DSKG}, adapting  from \cite{D99,DKMNX09,DM03}.   In Section \ref{sec proof}, we first employ an argument from Chen and Lee \cite{CL25} to reduce the problem’s generality, confining the proofs to the critically damped case. Then, based on this reduction, we adapt techniques from \cite{HOO2024} to prove Theorem \ref{thm error 1} within the simplified setting. 
 Finally,  in Section \ref{QVPE}, we develop a quadratic variation analysis for the nonlinear stochastic wave equation and construct the corresponding statistical estimators.
      
 \section{Properties of the  solution}\label{Sec 2}
 \subsection{Green function}
       Taking the Fourier transform in the  spatial variable,    $\mathscr F{G}(t)(\xi):=\hat G(t, \xi)$, the Fourier transform of $G(t, \cdot)$, satisfies the ordinary differential equation
\begin{equation}\label{eq:ODE}
\partial_t^2\hat G  + a\partial_t\hat G  + \bigl(\xi^2+m^2\bigr)\hat G = 0,
\end{equation}
subject to the initial conditions
$$
\hat G(0, \xi)=0,\qquad \partial_t\hat G (0, \xi)=1,
$$  
which  follow from  integrating the original equation in time around $t=0$ and using the fact that $G(t,\cdot)=0$ for $t<0$. 

Solving   \eqref{eq:ODE}  yields 
  \begin{equation}\label{Fourier G}
    \mathscr F{G}(t)(\xi)= 
\begin{cases}
\displaystyle
e^{-\frac{at}{2}}\frac{ \sin \Bigl(t\sqrt{\xi^2+m^2-a^2/4}\,\Bigr)}
     {\sqrt{\xi^2+m^2-a^2/4}}, 
& \text{if } \xi^2 > \frac{a^2}{4} - m^2;\\[8pt]
\displaystyle
e^{-\frac{at}{2}}\, t,  
& \text{if } \xi^2 = \frac{a^2}{4} - m^2;\\[8pt]
\displaystyle
e^{-\frac{at}{2}} \frac{ \sinh \Bigl(t\sqrt{\frac{a^2}{4}-m^2-\xi^2}\,\Bigr)}
     {\sqrt{\frac{a^2}{4}-m^2-\xi^2}},  
& \text{if } \xi^2 < \frac{a^2}{4} - m^2.
\end{cases}
\end{equation}
Note that the first   expression in \eqref{Fourier G} formally contains the other two, since  
$$
\lim_{u\rightarrow0}\frac{\sin(u)}{u}=1\ \ \ \ \text{and } \sin(iu)=i \sinh(u). 
$$
Consequently, we obtain the following unified formula via analytic continuation:
 \begin{equation}\label{eq unif}
  \mathscr F{G}(t)(\xi)= e^{-\frac{at}{2}}\frac{ \sin \Bigl(t\sqrt{\xi^2+m^2-a^2/4}\,\Bigr)}
     {\sqrt{\xi^2+m^2-a^2/4}},
\end{equation}
where the square root and the sine function are understood in the sense of analytic continuation on the complex plane.   For the detailed derivation, we refer to \cite[Example 7]{D99} and \cite[pp. 50--51]{DKMNX09}.

 \begin{lemma}\label{lem G1}There exists a constant $c>0$ such that,  for any $T>0$, 
  \begin{align}\label{eq G110}
  \sup_{t\in [0,T]}\int_{\mathbb R}|\mathscr F G(t)(\xi)|^2d\xi\le c\left(1+T^2\right).
  \end{align}
  \end{lemma}
  \begin{proof}  
  This estimate  is essentially contained in the proof of Lemma 2.2 in \cite{CL25}. For completeness, we  provide the details.
  
    Using the unified representation  \eqref{Fourier G}, we have 
  \begin{align*}
  \int_{\mathbb R}|\mathscr  F G(t)(\xi)|^2d\xi= &\,  \int_{\xi^2<\frac{a^2}{4}-m^2} e^{-a t}\frac{\sinh^2\left(t\sqrt{\xi^2+m^2-a^2/4}\right)}{\xi^2+m^2-a^2/4} d\xi\\ 
  &\,\, + \int_{\xi^2>\frac{a^2}{4}-m^2} e^{-at}\frac{\sin\left(t \sqrt{\xi^2+m^2-a^2/4}\right)}{\xi^2+m^2-a^2/4} d\xi\\
   =:  &\, I_1+I_2.
  \end{align*}
  
  Set $K:=\sqrt{\left(\frac{a^2}{4}-m^2\right)\vee 0}$. 
  Since $\frac{\sinh(x)}{x}$ is uniformly bounded on any compact interval of $\mathbb R$, we have 
  \begin{align*}
  I_1= \,  \int_{\xi^2\le \frac{a^2}{4}-m^2} e^{-at} t^2\frac{\sinh^2(t\sqrt{\xi^2+m^2-a^2/4}}{t^2(\xi^2+m^2-a^2/4)} d\xi 
 \lesssim   \,      t^2. 
  \end{align*} 
  Here and throughout, we write $a\lesssim b$  if there exists a positive constant $C$, independent of the relevant parameters, such that $a\le C b$.
 
 The second term $I_2$  is further decomposed  into $I_{2,1}+I_{2,2}$, where 
  \begin{align*}
  I_{2,1}:=&\,   \int_{ K<|\xi| <2K} e^{-at}\frac{\sin\left(t \sqrt{\xi^2+m^2-a^2/4}\right)}{\xi^2+m^2-a^2/4} d\xi,\\
   I_{2,2}:=&\,   \int_{|\xi| \ge 2K} e^{-at}\frac{\sin\left(t \sqrt{\xi^2+m^2-a^2/4}\right)}{\xi^2+m^2-a^2/4} d\xi.
  \end{align*}
     Using $\left| \frac{\sin(x)}{x}\right|\le 1$ and the same trick used for $I_1$, we have
  $$
  I_{2,1}\lesssim  t^2. 
  $$
  For $I_{2,2}$,  since $|\sin(x)|\le 1$, we have 
  $$
  I_{2,2}\lesssim  \,  \int_{|\xi| \ge 2K} \frac{1}{\xi^2-K^2} d\xi 
 \lesssim    \,     \int_{|\xi| \ge 2K} \frac{1}{\xi^2 } d\xi =\frac{1}{K}. 
  $$
  
  The proof is complete. 
  \end{proof}

\subsection{The existence and uniqueness of the solution} 
 In this part, we establish the well-posedness of the solution to the equation \eqref{eq:DSKG}.

\begin{proposition}\label{prop solution Holder}  
Suppose that $F$ is Lipschitz continuous.  Then the equation \eqref{eq:DSKG} admits a   unique mild solution $u$   given by 
\begin{equation}\label{eq:mild sol}
    \begin{split}
    u(t, x)=    \int_0^t \int_{\mathbb R} G(t-s, x-y) F(u(s, y)) W(ds, dy).
     \end{split}
\end{equation}
  Moreover,  the solution satisfies  
\begin{itemize}
 \item[(a)] For every  $T>0$ and  $1 \le p<\infty$,  
\begin{align}\label{eq moment bound}
\sup_{0\le t\le T}\sup_{x\in \mathbb R}\mathbb E\left[|u(t,x)|^p\right]<\infty. 
\end{align}
 
\item[(b)] For every  $T>0$  and $1 \le p<\infty$,  there exists a constant $c(p, T)>0$ such that 
\begin{align}\label{eq Holder}
\mathbb E\left[\left|u(t, x)-u(s, y)\right|^p \right]\le c(p, T)\left[|t-s|^{\frac{1}{2}}+|x-y|^{\frac{1}{2}} \right]^{p},
\end{align}
for all $t, s\in [0,T]$ and $x, y\in \mathbb R$. 
\end{itemize}
  \end{proposition}
 \begin{proof}  The existence and uniqueness of the solution follow from a standard Picard iteration. Once the moment bound \eqref{eq moment bound} is established,   the H\"older continuity  \eqref{eq Holder} can be derived via a  classical argument based on  \cite[Lemmas 2.2, 2.3, 2.4]{CL25} (see, e.g., \cite[Section 5]{D99}, \cite[Theorem 9]{DM03}, \cite[pp. 59--61]{DKMNX09},  \cite[Lemma 4.2.4]{DS26}, \cite{KS23}).     For  brevity, we  outline  only  the main  steps for  the existence.
      
Set  $u_0 (t, x): = 0$. For $n\ge0$,  assume inductively  that    an adapted, mean-square continuous process
  $\{u_n(s, \cdot)\}_{0\le s\le T}$ taking   values in $L^2(\mathbb R)$ has been defined.  We then define   
\begin{equation}\label{iter n}
u_{n+1} (t, x): =   \int_0^t \int_\R G(t-s, x-y) F(u_{n}(s,y)) W(ds, dy). 
\end{equation}
Because $F$  is   Lipschitz continuous,  $\{F(u_{n}(t,\cdot))\}_{t\in [0,T]} $ remains  adapted and mean-square continuous.  Therefore,   by   \cite[Lemma  6 and Theorem 6]{DM03}   together  with  Lemma \ref{lem G1}, the stochastic integral in \eqref{iter n} is well-defined.

 We first verify a uniform moment bound:  for every  $p\ge2$, 
 \begin{equation}\label{iter n2}
 \sup_{n\ge0}\sup_{0\le t\le T}\sup_{x\in\mathbb R}\mathbb E\left[\left|u_{n}(t, x) \right|^p\right]<+\infty. 
 \end{equation}
 Applying  the  Burkholder-Davis-Gundy (BDG, for short) inequality, H\"older's inequality, and the Lipschitz property of $F$, we have
 \begin{align*}
 & \mathbb E\left[\left|u_{n+1}(t, x) \right|^p\right]\\
 \le &\, c\mathbb E\left[\left|\int_0^t   \int_{\mathbb R} G(t-s, x-y)^2 F(u_n(s, y))^2 ds dy\right|^{\frac{p}{2}}  \right]\\
 \le &\,  c \left(\int_0^t ds \int_{\mathbb R} dy G(t-s, x-y)^2 \right)^{\frac{p}{2}-1} \times \int_0^tds\int_{\mathbb R} dy G^2(t-s, x-y)\mathbb E\left[1+\left|u_n(s, y)\right|^p \right].
 \end{align*}
 Using  Lemma \ref{lem G1},  we have 
    \begin{align*}
 \mathbb E\left[\left|u_{n+1}(t, x) \right|^p\right] \le &\, c(T)\int_0^t    \left(1+\sup_{y\in \mathbb R} \mathbb E\left[\left|u_n(s, y)\right|^p\right]\right)ds.
  \end{align*}
Define 
  $$
  M_n(t):=\sup_{y\in \mathbb R}\mathbb E\left[\left|u_n(t, y)\right|^p\right].
  $$
The above inequality becomes 
  \begin{align}\label{eq Grown 3}
  M_{n+1}(t)\le c(T) \int_0^t\left(1+M_n(s)\right)ds. 
  \end{align}
  An application of Gronwall’s lemma (see, e.g., \cite[p. 26]{DKMNX09}) then yields 
  $$
  \sup_{n\ge 1}\sup_{0\le t\le T}M_{n}(t)<+\infty,
  $$
 which proves  \eqref{iter n2}.
  
  We  now turn  to the  $L^p(\Omega)$-convergence of the Picard  scheme.   An argument similar to that leading to    \eqref{eq Grown 3} gives 
 \begin{equation}
 \begin{split}
& \sup_{x\in \mathbb R} \mathbb E \left[\left|u_{n+1}(t, x)-u_{n}(t, x)\right|^p\right]\\
 \le&\,  c(T) \int_0^t   \sup_{y\in \mathbb R} \mathbb E\left[\left|u_{n}(s, y)-u_{n-1}(s, y)\right|^p\right]  ds.
\end{split}
\end{equation}
By Gronwall’s lemma, the sequence  $\{u_n(t, x)\}_{n \ge 1}$ converges in $L^p(\Omega, \mathcal F, \mathbb P)$  uniformly in $x \in  \mathbb R$  to a limit $u(t, x)$.  Moreover, as in the proof of     \cite[Theorem 9]{DM03}, the limiting process   $\{u(t,\cdot)\}_{ 0 \le  t \le  T }$ is    adapted and   mean-square continuous.  
   
       The proof is complete. 
  \end{proof}

\section{Local linearization  from a bi-parameter viewpoint} \label{sec proof}
  \subsection{Reduction to the critical damping case and the proof of Theorem \ref{thm error 1}}\label{s:reduction}
  We begin by  rewriting  the equation \eqref{eq:DSKG} as 
\begin{align}\label{eq:DSKG+b}
    \begin{cases}
        \left(\DAlambert + a\partial_t + \frac{a^2}{4}\right) u(t, x) = F(u(t, x))\dot{W} (t, x) + b(u(t, x)), \quad t>0, \, x\in \R,\\
        u(0,x) = \partial_t u(0,x) = 0,
    \end{cases}
\end{align}
where   $      b(x):= \left(\frac{a^2}{4}-m^2\right) x$. Since $b$ is globally Lipschitz, the  standard existence and uniqueness theory   ensures  that the solution   to \eqref{eq:DSKG} coincides with the unique mild solution to \eqref{eq:DSKG+b}, given by  the  integral equation 
\begin{equation}\label{eq:DSKG+b:mild}
\begin{split}
    u(t,x) =&    \int_0^t \int_\R \Gamma(t-s,x-y) b(u(s,y)) dsdy\\
    &\, +  \int_0^t \int_\R \Gamma(t-s, x-y) F(u(s, y))W(ds, dy).
\end{split}
\end{equation}
Here, $\Gamma$ denotes the fundamental solution to the operator $\DAlambert + a \partial_t + a^2/4$. From  \eqref{Fourier G}, we directly obtain  
 \begin{equation*} 
     \mathscr F \Gamma(t)(\xi) = \frac{e^{-\frac{at}{2}}\sin(t|\xi|)}{|\xi|}, 
     \end{equation*} 
     and   
      \begin{equation}\label{G: critical2}
     \Gamma(t,x) = \frac{1}{2} e^{-\frac{at}{2}} \1_{\{|x|< t\}}.   
\end{equation} 

Define the linear and critical components of the solution as
    \begin{align}  
   u_L(t, x) := &\, \left(\frac{a^2}{4}-m^2\right) \int_0^t \int_{\mathbb R} \Gamma(t-s,x-y) u(s,y) dsdy, \label{Eq uL}\\
     u_C(t,x) :=&\,   \int_0^t \int_{\mathbb R} \Gamma(t-s, x-y) F(u(s, y)) W(ds, dy). \label{Eq uC} 
\end{align}
Observe that $u_C$ is   exactly  the   mild solution to \eqref{eq:DSKG}  in the critically damped case  $m^2 = \frac{a^2}{4}$.  Substituting  the explicit form \eqref{G: critical2}  into   \eqref{Eq uC} yields 
  \begin{equation}\label{eq solution uc}
\begin{split}
   u_C(t,x) =\, \frac{1}{2}\iint_{\Delta(t,x)} e^{-\frac{a(t-s)}{2}}F(u(s,y))W(ds,dy),
   \end{split}
\end{equation}
 where  
  \begin{align}\label{eq Delta tx}
\Delta(t, x):=\left\{(s,y)\in \mathbb R_+\times \mathbb R: 0\le s\le t, \, |x-y|\le t-s \right\}.
\end{align}
This  region is  illustrated as  the shaded area in Figure \ref{fig1}.

 \begin{figure}[htbp]
\centering
\begin{minipage}[c]{0.48\textwidth}
\centering
\begin{tikzpicture} 
\path [fill=lightgray] (0,-2) -- (0,4) -- (3,1);
\draw [->] (0,-2.5) -- (0,4.6);
\draw [->] (0,0) -- (4,0);
\draw (0, -2) -- (0, 4) -- (3,1) -- (0,-2);
\draw [dashed] (3,0) node [below] {$t$} -- (3,1);
\draw [dashed] (0,1) node [left] {$x$} -- (3,1);
\draw [fill] (3,1) circle [radius=.05];
\draw [fill] (0,4) circle [radius=.05] node [left]{$x+t$};
\draw [fill] (0,-2) circle [radius=.05] node [left]{$x-t$};

\draw [fill] (0,4.6)   node [above]{$y$};
\draw [fill] (4.0,0)   node [right]{$s$};

  \end{tikzpicture}
\caption{}\label{fig1}
\end{minipage}
\hfill
\begin{minipage}[c]{0.48\textwidth}
\centering
\begin{tikzpicture}[scale=0.5]

\draw [->] (0,-6) -- (0,8) node [above] {$y$};
\draw [->] (0,0) -- (8,0) node [right] {$s$};

\draw [thick,domain=0:8/sqrt(2)] plot (\x, {8/sqrt(2)-\x+1});      
\draw [thick,domain=0:5.5/sqrt(2)] plot (\x, {3/sqrt(2)-\x+1});    
\draw [thick,domain=0:8/sqrt(2)] plot (\x, {-8/sqrt(2)+\x+1});     
\draw [thick,domain=0:5.5/sqrt(2)] plot (\x, {-3/sqrt(2)+\x+1});   

\coordinate (O) at (0,0);          
 
\coordinate (G) at ({3/sqrt(2)},1);                                     
\coordinate (A) at ({8/sqrt(2)},1);                                
\coordinate (B) at (0,{8/sqrt(2)+1} );                                
\coordinate (C) at (0,{3/sqrt(2)+1} );                                
\coordinate (F) at ({5.5/sqrt(2)},{8/sqrt(2)-5.5/sqrt(2)+1});        
\coordinate (E) at ({5.5/sqrt(2)},{3/sqrt(2)-5.5/sqrt(2)+1});        
\coordinate (H) at (0,{-3/sqrt(2)+1});                               
\coordinate (I) at (0,{-8/sqrt(2)+1});                               

\fill[blue!40, opacity=0.4] (C) -- (G)-- (F) -- (B) -- cycle;

\fill[green!40, opacity=0.4] (F)-- (G) -- (E) -- (A) -- cycle;

\

\fill[yellow!40, opacity=0.4] (C) -- (G) -- (H) -- cycle;
\fill[red!40, opacity=0.4] (H) -- (G) -- (E) --(I) -- cycle;
 
\fill[blue] (O) circle (2pt);
\node[above left] at (O) {$O$};









\node[black] at (1.5, 3.5) {$D_1$};
\node[black] at (1.5, -1.5) {$D_2$};
\node[black] at (0.75, 1.25) {$D_3$};
\node[black] at (3.8, 1.25) {$D_4$};
 \node[black] at ({8/sqrt(2)+1},1.4) {$(t, x)$};
 
 \end{tikzpicture}
\caption{}\label{fig2}
\end{minipage}
\end{figure}

\begin{proof}[Proof of Theorem \ref{thm error 1}]      
 By symmetry,  it suffices  to prove  \eqref{eq error} for $R_{\varepsilon}(\tau, \lambda)$ with $\tau>0$, $\lambda\ge -\tau$, and $\varepsilon>0$.   Define 
\begin{equation}\label{SWE vL}
v_L(\tau, \lambda) := u_L\left(\frac{\tau+\lambda}{\sqrt{2}}, \frac{-\tau+\lambda}{\sqrt{2}}\right),  \  \ v_C(\tau, \lambda) := u_C\left(\frac{\tau+\lambda}{\sqrt{2}}, \frac{-\tau+\lambda}{\sqrt{2}}\right).
\end{equation}
From Propositions \ref{lem:rec_incr} and \ref{lem:rec_incr VC} below, we obtain
   $$
      \left\|\delta_{\varepsilon}^{(1)}\delta_{\varepsilon}^{(2)}  v_L(t, x) \right\|_{L^p(\Omega)}    +
\left \| \delta_{\varepsilon}^{(1)}\delta_{\varepsilon}^{(2)}v_C(\tau, \lambda)-F(v(\tau, \lambda))\delta_{\varepsilon}^{(1)}\delta_{\varepsilon}^{(2)}V(\tau, \lambda)\right\|_{L^p(\Omega)} \le c    \varepsilon^{\frac{3}{2}}.
$$
Applying the Minkowski inequality then yields the required estimate in Theorem \ref{thm error 1}.
 
               The proof is complete.            
                       \end{proof}

 \subsection{Increment Estimates for  the drift term}
  \begin{lemma}\label{lem G increm} For any $p\ge 1$ and sufficiently small   $\varepsilon>0$, 
\begin{equation}\label{eq:kernel_increment}
\begin{split}
& \int_0^t   \int_{\mathbb R}        \Bigg|  \Gamma(t-s, x-y)-\Gamma\left(t-\frac{\varepsilon}{\sqrt{2}}-s,x+\frac{\varepsilon}{\sqrt{2}}-y\right)\\
  &\,\,\,\,   -\Gamma\left(t-\frac{\varepsilon}{\sqrt{2}}-s, x-\frac{\varepsilon}{\sqrt{2}}-y\right)+\Gamma\left(t- \sqrt{2} \varepsilon-s, x-y\right)  \Bigg|^{p} dsdy \\
  =&\,  \frac{\varepsilon^{2}}{2^{p}}  +O\left(\varepsilon^{3}+\varepsilon^{1+p}\right). 
     \end{split}
  \end{equation}
\end{lemma}   
\begin{proof}  The argument follows the method of \cite[Lemma 5.1]{CL25}, which treats  the case  $p=2$. 
Recall the domain $\Delta(t, x)$ defined  in  \eqref{eq Delta tx}.     Define the subsets 
    \begin{equation}\label{eq Domain}
    \begin{split} 
    &  D_1:= \, \left( \Delta(t, x)\setminus \left\{\left|x-\frac{\varepsilon}{\sqrt{2}}-y\right|<t-\frac{\varepsilon}{\sqrt{2}}-s\right\}\right)\cap\left\{\left|x+\frac{\varepsilon}{\sqrt{2}}-y\right|<t-\frac{\varepsilon}{\sqrt{2}}-s\right\}, \\
      &   D_2:= \, \left( \Delta(t, x)\setminus \left\{\left|x+\frac{\varepsilon}{\sqrt{2}}-y\right|<t-\frac{\varepsilon}{\sqrt{2}}-s\right\}\right) \cap \left\{\left|x-\frac{\varepsilon}{\sqrt{2}}-y\right|<t-\frac{\varepsilon}{\sqrt{2}}-s\right\},\\
      &   D_{3}:=  \,  \Delta(t, x)\cap \left\{ \left|x+\frac{\varepsilon}{\sqrt{2}}-y\right|<t-\frac{\varepsilon}{\sqrt{2}}-s\right\} \cap \left\{\left|x-\frac{\varepsilon}{\sqrt{2}}-y\right|<t-\frac{\varepsilon}{\sqrt{2}}-s\right\},\\
    &  D_4:=  \,  \Delta(t, x)\setminus (D_1\cup D_2\cup D_{3}).
    \end{split}
    \end{equation}
    An illustration is given in Figure \ref{fig2}.

Denote the shifted kernels by 
   \begin{equation}
   \begin{split} 
    \Gamma_1(t, x):= &\, \Gamma\left(t-\frac{\varepsilon}{\sqrt{2}},x+\frac{\varepsilon}{\sqrt{2}}\right), \\
   \Gamma_2(t, x):=& \, \Gamma\left(t-\frac{\varepsilon}{\sqrt{2}}, x-\frac{\varepsilon}{\sqrt{2}}\right),  \\
 \Gamma_3(t, x):= & \, \Gamma\left(t-\sqrt{2}\varepsilon, x\right). 
   \end{split}
   \end{equation}
The following observations about these sets will be used:
    \begin{itemize}
        \item $\Delta(t, x)$ is a disjoint union of $D_i, i=1, \cdots, 4$.  It is the support of $\Gamma$.
                  \item For $i=1,2$,  $D_i, $ has area smaller than $\varepsilon  T$ and only $\Gamma$ and $\Gamma_i$ are nonzero on it. 
        \item $D_3$ has area  at most  $T^2/2$ and $\Gamma, \Gamma_1, \Gamma_2$ and $\Gamma_{3}$ are all nonzero on it.
         \item $D_4$ has area $\varepsilon^2$ and only $\Gamma$ is nonzero on it.
    \end{itemize}
    
    Accordingly, we decompose the integral as 
    $$  I_1+I_2+I_3+I_4,$$
    where 
    \begin{align*}
     I_1:=&\, \iint_{D_1} (\Gamma-\Gamma_1)^{p} (t-s, x-y)dsdy,\\
      I_2:= &\,  \iint_{D_2} (\Gamma-\Gamma_2)^{p}(t-s, x-y) dsdy, \\
     I_{3}:=&\, \iint_{D_{3}} (\Gamma-\Gamma_1-\Gamma_2+\Gamma_{3})^{p}(t-s, x-y) dsdy,\\
       I_{4}:=&\, \iint_{D_4} \Gamma^{p}(t-s, x-y) dsdy.
    \end{align*}
             
    For $I_1$,  we have by Taylor's formula, 
     \begin{align*}
        I_1  \lesssim  &\,  \varepsilon  T\max_{s\in \left[0,t-\frac{\varepsilon}{\sqrt{2}}\right]}\left|e^{-\frac{a}{2}(t-s)}-e^{-\frac{a}{2}\left(t-\frac{\varepsilon}{\sqrt{2}}-s\right)}\right|^{p}\\
    \lesssim   &  \,     \varepsilon^{1+p} +o\left( \varepsilon^{1+p} \right).
    \end{align*}
     The same bound holds for $I_2$.

For  $I_{3}$, again using Taylor's formula,  
    \begin{align*}
        I_{3}  \lesssim   &\,    \max_{s\in [0,t-\sqrt{2} \varepsilon]}\left|e^{-\frac{a}{2}(t-s)} - e^{-\frac{a}{2}\left(t-\frac{\varepsilon}{\sqrt{2}}-s\right)}- e^{-\frac{a}{2}\left(t-\frac{\varepsilon}{\sqrt{2}}-s\right)} +e^{-\frac{a}{2}\left(t-\sqrt{2}\varepsilon -s\right)} \right|^{p}\\
        \lesssim  &\,    \left|1+e^{\frac{a\varepsilon}{\sqrt{2}} }- 2e^{\frac{a\varepsilon}{2 \sqrt{2}} } \right|^{p} \lesssim   \,  \varepsilon^{2p}. 
     \end{align*}

    For $I_{4}$, expand $\Gamma$ for small $|t-s|$:
    $$
   \Gamma(t-s, x-y)  = \frac12  -\frac{a(t-s)}{4} +o\left(|t-s|\right).
    $$
      Since the area of $D_4$ is $\varepsilon^2$ and   $|t-s|\le  \varepsilon/\sqrt{2}$ for  any $(s,y)\in D_4$,
   we have
      \begin{align*}
    I_{4}  =  \frac{\varepsilon^{2}}{2^p} + O\left(\varepsilon^{3}\right).
    \end{align*}
    
Summing the four estimates gives \eqref{eq:kernel_increment}. The proof is complete. 
\end{proof}

 \begin{proposition}\label{lem:rec_incr}  
     For any $p\ge1$,  there exists a constant $c>0$ such that 
 \begin{equation}
        \begin{split}
         \sup_{t\in [0,T], x\in \mathbb R} \left\|\delta_{\pm\varepsilon}^{(1)}\delta_{\varepsilon}^{(2)}  v_L(t, x) \right\|_{L^p(\Omega)}           \le  \,  c \varepsilon^{2}.
    \end{split}
    \end{equation}
         \end{proposition}
                 \begin{proof}  The proof  follows   \cite[Section 3]{AT21} and \cite[Lemma 2.7]{W2024}. By  symmetry,  it   suffices to show 
        \begin{equation}\label{eq uL est alt}
        \begin{split}
       &  \left\|u_L(t,x)-u_L\left(t-\frac{\varepsilon}{\sqrt{2}},x+\frac{\varepsilon}{\sqrt{2}}\right)-u_L\left(t-\frac{\varepsilon}{\sqrt{2}}, x-\frac{\varepsilon}{\sqrt{2}}\right)+u_L\left(t- \sqrt{2} \varepsilon, x\right)  \right\|_{L^p(\Omega)}\\
          \le &\,  c \varepsilon^{2}.
    \end{split}
    \end{equation}
         
         Using the mild formulation of $u_L$ and   H\"older's inequality,    we have that   for any $p>1$,
 \begin{align*}
 & \mathbb E\Bigg[\int_0^t  ds \int_{\mathbb R}  dy \,   \left| \Gamma(t-s, x-y) - \Gamma \left(t - \frac{\varepsilon}{\sqrt{2}} - s, x + \frac{\varepsilon}{\sqrt{2}} - y\right) \right. \\
&\qquad\qquad \qquad \left. - \Gamma \left(t - \frac{\varepsilon}{\sqrt{2}} - s, x - \frac{\varepsilon}{\sqrt{2}} - y\right) 
+ \Gamma\left(t - \sqrt{2}\varepsilon - s, x - y\right) \right|   |u(s, y)|\Bigg]\\
\le &\Bigg[\int_0^t  ds \int_{\mathbb R}  dy \,   \left| \Gamma(t-s, x-y) - \Gamma \left(t - \frac{\varepsilon}{\sqrt{2}} - s, x + \frac{\varepsilon}{\sqrt{2}} - y\right) \right. \\
&\qquad\qquad \qquad \left. - \Gamma \left(t - \frac{\varepsilon}{\sqrt{2}} - s, x - \frac{\varepsilon}{\sqrt{2}} - y\right) 
  + \Gamma\left(t - \sqrt{2}\varepsilon - s, x - y\right) \right|  \Bigg]^{1-\frac1p}\\
& \times \Bigg[\int_0^t  ds \int_{\mathbb R}  dy \,   \left| \Gamma(t-s, x-y) - \Gamma \left(t - \frac{\varepsilon}{\sqrt{2}} - s, x + \frac{\varepsilon}{\sqrt{2}} - y\right) \right. \\
&\qquad\qquad \qquad \left. - \Gamma \left(t - \frac{\varepsilon}{\sqrt{2}} - s, x - \frac{\varepsilon}{\sqrt{2}} - y\right) 
+ \Gamma\left(t - \sqrt{2}\varepsilon - s, x - y\right) \right|  \cdot \mathbb E\left[|u(s,y)|^p\right] \Bigg]^{\frac1p},
 \end{align*}  
 where $\Gamma_1, \Gamma_2, \Gamma_3$ denote the shifted kernels as in Lemma \ref{lem G increm}.
           Applying  the uniform moment bound  \eqref{eq moment bound} and  Lemma \ref{lem G increm} (with $p=1$), we obtain \eqref{eq uL est alt}.
                The proof is complete. 
  \end{proof}

   The following  estimate follows by applying  It\^o-Walsh isometry and  adapting the proof  of  Lemma \ref{lem G increm} (or  Lemma 5.1 in   \cite{CL25}); we omit   the details here. 
   \begin{corollary}\label{coro V}    Let $V$ be defined  by \eqref{SWE 3}.     Then  for any fixed $M, N>0$,  there exists a constant $c>0$ such that  for any   $\varepsilon>0$, 
  \begin{equation}\label{eq V est}
        \begin{split}
         \sup_{\tau\in [0,M]}\sup_{\lambda\in [-\tau, N]}  \left|  \left\|    \delta_{\pm\varepsilon}^{(1)}\delta_{\varepsilon}^{(2)}V(\tau, \lambda) \right\|_{L^2(\Omega)}-\frac{\varepsilon}{2}\right| \le c  \varepsilon^{\frac{3}{2}}. 
    \end{split}
    \end{equation}      
    \end{corollary}

 \subsection{Increment Estimates for  the diffusion term}
      We prove the following result  by adopting a strategy analogous to that  employed  in Theorem 1.3 of \cite{HOO2024}.
        
\begin{figure}
\begin{tikzpicture}[scale=0.88]
\draw [->] (0,-3.75) -- (0,4) node [above] {$y$};
\draw [->] (0,0) -- (6.8,0) node [right] {$s$};
\draw [->] (0,0) -- (3.6,3.6) node [above right] {$\lambda$};
\draw [->] (0,0) -- (3.7,-3.7) node [right] {};

\draw [thick,domain=5.5/sqrt(2):8/sqrt(2)] plot (\x, {8/sqrt(2)-\x});
\draw [dashed,domain=4/sqrt(2):5.5/sqrt(2)] plot (\x, {8/sqrt(2)-\x});

\draw [thick,domain=3/sqrt(2):5.5/sqrt(2)] plot (\x, {3/sqrt(2)-\x});
\draw [dashed,domain=1.5/sqrt(2):4/sqrt(2)] plot (\x, {3/sqrt(2)-\x});

\draw [thick,domain=5.5/sqrt(2):8/sqrt(2)] plot (\x, {-8/sqrt(2)+\x});
\draw [dashed,domain=4/sqrt(2):5.5/sqrt(2)] plot (\x, {-8/sqrt(2)+\x});

\draw [thick,domain=3/sqrt(2):5.5/sqrt(2)] plot (\x, {-3/sqrt(2)+\x});
\draw [dashed,domain=1.5/sqrt(2):4/sqrt(2)] plot (\x, {-3/sqrt(2)+\x});

\draw [thick,domain=-2.5/sqrt(2):2.5/sqrt(2)] plot ({5.5/sqrt(2)}, \x); 

\draw [fill] ({4/sqrt(2)} ,{4/sqrt(2)}) circle [radius=.05] 

node [xshift=-0.6cm, yshift=0.25cm]{$\lambda+\varepsilon$};

\draw [fill] ({1.5/sqrt(2)},{1.5/sqrt(2)}) circle [radius=.05] 
node [xshift=-0.6cm, yshift=0.25cm]{$\lambda$};
\draw [fill] ({1.5/sqrt(2)},{-1.5/sqrt(2)}) circle [radius=.05] node [below left]{$\tau$};
\draw [fill] ({4/sqrt(2)},{-4/sqrt(2)}) circle [radius=.05] node [below left]{$\tau+\varepsilon$};
\node at (3.35,0.3) {$D_{4,1}$};
\node at (4.35,0.3) {$D_{4, 2}$};

\draw [fill] ({3/sqrt(2)},{0}) circle [radius=.05] node[below=0.6mm]{$P_1$};
\draw [fill] ({5.5/sqrt(2)},{2.5/sqrt(2)}) circle [radius=.05] node [above=0.6mm]{$P_3$};
 
\draw [fill] ({5.5/sqrt(2)},{-2.5/sqrt(2)}) circle [radius=.05] node [below=0.6mm]{$P_2$};

\draw [fill] ({8/sqrt(2)},{0}) circle [radius=.05] node [below=0.6mm]{$P_4$};
   
\end{tikzpicture}
\caption{}\label{fig3}
\end{figure}

   \begin{proposition}\label{lem:rec_incr VC}  
    For any $p\ge1$,  there exists a constant $c>0$ such that  for any   $\varepsilon>0$,
\begin{align}
\left \| \delta_{\pm\varepsilon}^{(1)}\delta_{\varepsilon}^{(2)}v_C(\tau, \lambda)-F(v(\tau, \lambda))\delta_{\pm\varepsilon}^{(1)}\delta_{\varepsilon}^{(2)}V(\tau, \lambda)\right\|_{L^p(\Omega)} \le c    \varepsilon^{\frac{3}{2}}.
\end{align}
\end{proposition}
  \begin{proof}         We adopt the same domain decomposition  as in the  proof of Lemma   \ref{lem G increm}.   Let  $\Delta(t, x)$ be domain defined in \eqref{eq Delta tx},  and let   $D_i$ for  $i=1, \cdots, 4$  be the subsets   introduced  in     \eqref{eq Domain}.
  
As in  Lemma \ref{lem G increm}, we have
\begin{equation} 
\begin{split} 
 & \mathbb E\left[ \left|\delta_{\pm\varepsilon}^{(1)}\delta_{\varepsilon}^{(2)}v_C(\tau, \lambda)\right|^p\1_{\Delta\left(\frac{\tau+\lambda}{\sqrt{2}}, \frac{-\tau+\lambda}{\sqrt{2}}\right)\setminus D_4} \right]\\
  &\,\, +
    \mathbb E\left[ \left|F(v(\tau, \lambda))\delta_{\pm\varepsilon}^{(1)}\delta_{\varepsilon}^{(2)}V(\tau, \lambda)\right|^p\1_{\Delta\left(\frac{\tau+\lambda}{\sqrt{2}}, \frac{-\tau+\lambda}{\sqrt{2}}\right)\setminus D_4} \right]\\
     \lesssim &\,   \varepsilon^{2}.
    \end{split}  
\end{equation}
  Consequently, it suffices to prove  
  \begin{equation}\label{eq D4 est}
 \left\|\iint_{D_{4}} e^{-\frac{a}{2}(t-s)}\left[F(u(s,y))-F(u(P_1)) \right]W(ds,dy)\right\|_{L^p(\Omega)}   \lesssim\, \varepsilon^{\frac{3}{2}}. 
  \end{equation} 
  
  To establish \eqref{eq D4 est}, we adapt the method of  \cite[Theorem 1.3]{HOO2024}.   In order to ensure that   the integrand is  adapted, we further decompose  $D_4$ into two triangular subregions.   
  
   For   fixed $\tau>0, \lambda\ge -\tau$, and $\varepsilon>0$,  define  
\begin{align*}
D_{4,1}:= & \, \Bigg\{(s, y)\in \mathbb R_+\times \mathbb R:\,  \frac{\tau+\lambda}{\sqrt 2} <  \,  s\le \frac{\tau+\lambda+\varepsilon}{\sqrt 2},\\
 & \,\,\,\,\,\,\,\,\,\,\,    \,\,\,\,\,\,\,\,\,\,\,  \,\,\,\,  \,\,\,\, \,\,\,\,  \,\, \,\, -s+\sqrt{2}\lambda <  y < s-\sqrt{2}\tau \Bigg\},\\
D_{4,2}:=&  \, \Bigg\{(s,y)\in \mathbb R_+\times \mathbb R :\,  \frac{\tau+\lambda+\varepsilon}{\sqrt 2} <  \, s \le \frac{\tau+\lambda+2\varepsilon}{\sqrt 2},\\
 & \,\,\,\,\,\,\,\,\,\,\,    \,\,\,\,\,\,\,\,\,\,\, \,\,\,\,  \,\,\,\,  \,\,\,\, \,\,\,\,    s-\sqrt{2}(\tau+\varepsilon)\le \, y\le -s+\sqrt 2(\lambda+ \varepsilon) \Bigg\}.
\end{align*}
 The  sets $D_{4, 1}$ and $D_{4, 2}$ are  illustrated in Figure \ref{fig3}. Specifically,   $D_{4,1}$ is the triangular region with vertices 
\begin{equation}\label{eq P}
\begin{split}  
   P_1 :=&\, \left(\frac{\tau+\lambda}{\sqrt 2}, \frac{-\tau+\lambda}{\sqrt 2}\right),\\
     P_2:=&\,  \left(\frac{\tau+\lambda+\varepsilon}{\sqrt 2}, \frac{-\tau+\lambda-\varepsilon}{\sqrt 2}\right),\\
P_3:=&\, \left(\frac{\tau+\lambda+\varepsilon}{\sqrt 2}, \frac{-\tau+\lambda+\varepsilon}{\sqrt 2}\right),\\
  \end{split}  
\end{equation}
and $D_{4, 2}$ is the triangular region with vertices 
\begin{align*}
 P_2 , \, P_3,  \text{ and }\  P_4 :=&\,  \left(\frac{\tau+\lambda+2\varepsilon}{\sqrt 2}, \frac{-\tau+\lambda}{\sqrt 2}\right).
\end{align*}

We now decompose  the integral as follows:
\begin{align*}
 &  \frac{1}{2}\left(\iint_{D_{4,1}}+\iint_{D_{4, 2}} \right) e^{-\frac{a}{2}(t-s)}\left[F(u(s,y))-F(u(P_1)) \right]W(ds,dy)\\
=&\,  \frac{1}{2} \iint_{D_{4,1}} e^{-\frac{a}{2}(t-s)} \left[F(u(s, y))-F(u(P_1))\right]W(ds,dy) \\
&\,  +\frac{1}{2}   \left[F(u(P_2))- F(u(P_1)) \right] \iint_{D_{4,2}}  e^{-\frac{a}{2}(t-s)}W(ds,dy)\\
&\, +   \frac{1}{2} \iint_{D_{4,2}}e^{-\frac{a}{2}(t-s)}  \left[F(u(s,y))-F(u(P_2)) \right]W(ds,dy)\\
=:&\,  I_{4,1}+I_{4,2}+I_{4, 3}.
\end{align*} 

Applying   BDG's  inequality,    the estimate  \eqref{eq Holder},  and the Lipschitz continuity of $F$, we obtain 
\begin{align*}
 \mathbb{E}\left[\left|I_{4, 1}\right|^p\right]^{\frac{2}{p}}
 \lesssim &\,   \int_{\frac{\tau+\lambda}{\sqrt 2} }^ \frac{\tau+\lambda+\varepsilon}{\sqrt 2}ds \int_{-s+ \sqrt{2}\lambda}^{s -\sqrt{2}\tau } dy     
    e^{-a(t-s)} \left\| F(u(s,y)) - F(u(P_1)) \right\|_{L^p(\Omega)} ^2     \\
\lesssim &\,   \int_{\frac{\tau+\lambda}{\sqrt 2} }^ \frac{\tau+\lambda+\varepsilon}{\sqrt 2}ds\int_{-s + \sqrt{2}\lambda}^{s -\sqrt{2}\tau } dy     e^{-a(t-s)}
  \left( \left| s - \frac{\tau+\lambda}{\sqrt{2}} \right| + \left| y - \frac{-\tau+\lambda}{\sqrt{2}} \right| \right) \\
\lesssim &\,    \int_{\frac{\tau+\lambda}{\sqrt 2} }^ \frac{\tau+\lambda+\varepsilon}{\sqrt 2}ds   \left| s - \frac{\tau+\lambda}{\sqrt{2}} \right|^2 \\
= &\,  c \varepsilon^{3}. 
  \end{align*}

 A similar calculation yields  
\begin{align*}
 \mathbb{E}\left[\left|I_{4, 3}\right|^p\right]^{\frac{2}{p}}\lesssim   &\,    \varepsilon^{3}. 
   \end{align*}
 
   For the remaining term, we apply   H\"older's  inequality to obtain  
\begin{equation}
\begin{split}
 \mathbb E\left[|I_{4, 2}|^p\right]^{\frac{2}{p}} \le  &\,  \mathbb E\left[ \left| u(P_2)- u(P_1)\right|^{2p}\right]^{\frac{1}{p}}  \cdot \mathbb E\left[ \left|\iint_{D_{4, 2}} e^{-\frac{a}{2}(t-s)} W(ds,dy)\right|^{2p}\right]^{\frac1p}.
\end{split}
\end{equation} 
 By the Lipschitz continuity of $F$ together with   \eqref{eq Holder} and  \eqref{eq P}, we have 
 \begin{align*}
 \mathbb E\left[ \left| u(P_2)- u(P_1)\right|^{2p}\right]^{\frac{1}{p}}  \lesssim  \varepsilon.
 \end{align*}
Since $\iint_{D_{4, 2}} e^{-\frac{a}{2}(t-s)} W(ds,dy)$ is a centered  Gaussian random variable whose  variance is bounded by     
$$
   \int_{\frac{\tau+\lambda+\varepsilon}{\sqrt 2}}^{   \frac{\tau+\lambda+2\varepsilon}{\sqrt 2}}ds\int_{
  s-\sqrt{2}(\tau+\varepsilon)}^{  -s+\sqrt 2(\lambda+ \varepsilon)}dy e^{-a(t-s)}     \lesssim  \varepsilon^2,
$$
we obtain
\begin{align*}
\mathbb E\left[ \left|\iint_{D_{4, 2}} e^{-\frac{a}{2}(t-s)} W(ds,dy)\right|^{2p}\right]^{\frac{1}{p}} 
   \lesssim   \,     \varepsilon^2.
 \end{align*}

Combining the estimates for $I_{4,1}$, $I_{4,2}$  and $I_{4,3}$, we obtain the desired bound \eqref{eq D4 est}.   The proof is complete. 
\end{proof}

\section{Quadratic variations and parameter estimation}\label{QVPE}
Parameter    estimation for the stochastic wave equation driven by    additive noise has been extensively studied (see, e.g., \cite{AGT2022, KT2018, KTZ2018, Shev2023};  we also refer to  the monograph \cite{T2023} for   a comprehensive overview). In contrast,  the case of  multiplicative noise   remains less explored despite its greater practical relevance and theoretical challenges.

In this section, we analyze the quadratic variation of the solution and construct a consistent estimator for the diffusion parameter.
   \subsection{Quadratic  variations} \label{Sec QV}
     Recall the function $V$  defined in \eqref{SWE 3} and the    grid points $(t_i, \lambda_j)$,  $0\le i, j\le N$, defined in   \eqref{eq grid}.
   Define the second-order increment
  $$
\Delta_{i,j}(V) := V\left(\tau_{i+1}, \lambda_{j+1}  \right)-V\left(\tau_{i+1}, \lambda_{j}\right)-V\left(\tau_{i}, \lambda_{j+1}  \right) +V\left(\tau_{i}, \lambda_{j}    \right).
$$
    The  second-order  quadratic variation of $V$ is   
   \begin{align}\label{eq V V}
Q_{N}(V):=  \sum_{i=0}^{N-1}\sum_{j=0}^{N-1}\Delta_{i,j}(V)^2.
\end{align}

   We first  study  the asymptotic behavior of the sequence   $\{Q_{N}(V)\}_{ N\ge 1}$  as $N\rightarrow\infty$.
  	 	    	\begin{lemma}\label{lem linear app}   There exists a constant $c>0$ such that for any $N\ge 1$, 
		\begin{equation}\label{2f-10}
			\mathbb{E}\left[ \left| Q_{N}(V) -\frac14\right|^2   \right]  \le c N^{-1}.
					\end{equation} 
	\end{lemma}
 \begin{proof}
By  Corollary \ref{coro V},     the random variable  $ \Delta_{i,j}(V)$  follows  a  Gaussian distribution with mean zero and   variance  satisfying 
$$ \mathrm{Var} \left(\Delta_{i,j}(V)\right)= \frac{1}{4N^2} +O\left(N^{-3}\right).$$ 
Consequently, 
\begin{align}\label{eq Dij var}
\mathbb E\left[\Delta_{i,j}(V)^{2}\right] - \frac{1}{4N^2} \lesssim \,  N^{-3}.
\end{align}

  Using  standard moment bounds for  Gaussian random variables, we obtain that for all $i, j\geq 1$ and $p\in \mathbb N$,	 
	\begin{equation}\label{2f-2}
			\mathbb{E}\left[   \left( \Delta_{i,j}(V)^{2} - \frac{1}{4N^{2}}\right) ^{2p}\right]\lesssim  N^{-4p}.
		\end{equation} 
	
	Due to the  increment-independence  of the space-time white  noise $\{W(t, x)\}_{t\ge0, x\in\mathbb R}$,     	  the increments 
  $\Delta_{i,j}(V) $ and $\Delta_{i',j'}(V) $ are independent whenever  $(i, j)\neq (i', j')$.  Using \eqref{eq Dij var} and \eqref{2f-2},  we  have
       \begin{align*}
&   \mathbb{E}\left[   \left(\sum_{i=0}^{N-1}\sum_{j=0}^{N-1}  \left(     \Delta_{i,j}(V) ^{2} - \frac{1}{4N^2}\right)  \right)^{2}\right]   \\
  = & \,     \sum_{i=0}^{N-1}   \sum_{j=0}^{N-1}   \mathbb{E} \left[  \left(    \Delta_{i,j}(V)  ^{2} - \frac{1}{4N^2}  \right)^{2}\right]\\
  &\,\, + 2 \sum_{0\le i<i'\le N-1}    \sum_{0\le j<j'<N-1}    \mathbb{E} \left[       \Delta_{i,j}(V)^{2} - \frac{1}{4N^2}    \right]   \mathbb{E} \left[   \Delta_{i',j'}(V)^{2} - \frac{1}{4N^2}    \right]\\
\lesssim    &\,   N^{-1}.
    \end{align*} 
        The proof is complete. 
     \end{proof}
 
  Next, we give the proof of Proposition \ref{prop variation converge}.
 \begin{proof}[Proof of Proposition \ref{prop variation converge}]
 Recall the definitions of $v_L$ and $v_C$   in  \eqref{SWE vL}.
   First, applying Proposition \ref{lem:rec_incr} with $p=2$ and $\varepsilon = 1/N$ yields 
    $$\mathbb{E}\left[\Delta_{i,j}(v_L)^2\right] \le c  N^{-4}.$$
         Summing over all $i,j=0, \cdots, N-1$  gives
          $$\mathbb{E}[Q_N(v_L)] \le c  N^{-2},$$ 
    which is negligible compared to the leading-order terms.
    
Following  the same approach as  in \cite{LW25} and using the     increment approximation
 $$\Delta_{i,j}(v) \approx F(v(\tau_i, \lambda_j))\Delta_{i,j}(V),$$
  established  in Theorem \ref{thm error 1},   we  obtain the    estimate
   \begin{align}\label{eq quad vc}
  \mathbb E\left[\left| Q_{N}(v_C)- \frac14  \int_0^1\int_0^1 F^2\big(v(\tau, \lambda) \big) d\tau d\lambda \right|  \right]\lesssim\,  N^{- 1/2}.
  \end{align} 
 For the reader's convenience, we now provide a  detailed derivation of  \eqref{eq quad vc}.

   We decompose the  difference  as follows:
\begin{equation}
\begin{split}
&Q_{N}(v_C)-     \int_0^1\int_0^1 F^2\big(v(\tau, \lambda) \big) d\tau d\lambda\\
 =&\,  \sum_{i=0}^{N-1}\sum_{j=0}^{N-1}\left\{\Delta_{i,j}(v_C)^2 - F^2(v(\tau_i, \lambda_j)) \Delta_{i,j}(V)^2 \right\}\\
 & +\sum_{i=0}^{N-1}\sum_{j=0}^{N-1} F^2(v(\tau_i, \lambda_j)) \left\{ \Delta_{i,j}(V)^2   -\frac{1}{4N^2} \right\}\\
  &+\frac14 \left[\sum_{i=0}^{N-1}\sum_{j=0}^{N-1}  F^2(v(\tau_i, \lambda_j))  N^{-2} -  \int_0^1\int_0^1 F^2\big(v(\tau, \lambda) \big) d\tau d\lambda\right]\\
  =:&\, I_1+I_2+I_3.
 \end{split}
 \end{equation}
 
\noindent\textbf{Step 1:  Estimation of  $I_1$.}
  Using  the Cauchy-Schwarz inequality,     we have
 \begin{align*} 
&\mathbb E \left[\left|\Delta_{i,j}(v_C)^2 - F^2(v(\tau_i, \lambda_j)) \Delta_{i,j}(V)^2\right| \right]\\
\le &\,   \left( \mathbb E \left[\left|\Delta_{i,j}(v)-   F (v(\tau_i, \lambda_j)) \Delta_{i,j}(V)\right|^2 \right]\right)^{\frac12} \left( \mathbb E \left[\left|\Delta_{i,j}(v)+   F (v(\tau_i, \lambda_j)) \Delta_{i,j}(V)\right|^2 \right]\right)^{\frac12} .
 \end{align*}
 From Theorem \ref{thm error 1},  the first factor satisfies 
  \begin{align*}
  \left( \mathbb E \left[\left|\Delta_{i,j}(v_C)-   F (v(\tau_i, \lambda_j)) \Delta_{i,j}(V)\right|^2 \right]\right)^{\frac12} 
 \lesssim  \,   N^{-\frac{3}{2}}.
 \end{align*}
  For the second factor,    by  Minkowski's inequality,  H\"older's inequality,  \eqref{eq moment bound}, \eqref{2f-2},  and  Theorem \ref{thm error 1}, we have 
  \begin{align*}
&  \left( \mathbb E \left[\left|\Delta_{i,j}(v_C)+   F (v(\tau_i, \lambda_j)) \Delta_{i,j}(V)\right|^2 \right]\right)^{\frac12}\\
  \le &\, \left( \mathbb E \left[\Delta_{i,j}(v_C)^2 \right]\right)^{\frac12}+ \left( \mathbb E \left[  F (v(\tau_i, \lambda_j))^4\right]\right)^{\frac14}\cdot   \left( \mathbb E \left[ \Delta_{i,j}(V)^4 \right]\right)^{\frac14} \\
   \lesssim & \,  N^{-1}.
  \end{align*}
 Combining the two bounds   and summing over  $i, j=0, \cdots, N-1$, we conclude that 
 \begin{align*}
\mathbb E [I_1] \lesssim  N^{-\frac12}.
  \end{align*} 

\noindent\textbf{Step 2:  Estimation of  $I_2$.}
      Let    $\{\mathcal F_t, t\ge0\}$  be the filtration     generated by the space-time white noise $W$:
\begin{align}\label{eq sigma filed}
\mathcal F_t=\sigma\left\{W({\bf 1}_{[0,s]} \varphi), s\in [0, t], \varphi\in  C_0^{\infty}\right\} \vee \mathcal N, 
\end{align}
 where $\mathcal N$ denotes the  class of  $\mathbb P$-null sets in $\mathcal F$. 
      
      Consider   pairs $(i, j), (i', j')\in \mathbb N^2$  with  $i+j\ge i'+j'+2$. By \eqref{eq transform} and  the temporal independence of  $\{W(t, x)\}_{t\ge0, x\in\mathbb R}$,    the random variables    $v_C(\tau_{i'}, \lambda_{j'}), v_C(\tau_i, \lambda_j)$,  and $\Delta_{i',j'}(V)$ are  $\mathcal F_{\sqrt{2}(\tau_i+\lambda_j)}$-measurable,   while  $\Delta_{i,j}(V)$   is independent of   $\mathcal F_{\sqrt{2}(\tau_i+\lambda_j)}$. 
              Taking the conditional expectation with respect to    $\mathcal F_{\sqrt{2}(\tau_i+\lambda_j)}$ and using  \eqref{eq Dij var}, we obtain
      \begin{align}\label{eq Vij var}
      \mathbb E\left[    \Delta_{i,j}(V)^2   -\frac{1}{4N^2} \Big|\mathcal F_{\sqrt{2}(\tau_i+\lambda_j)}\right]  \lesssim \,  N^{-3}.
      \end{align}
      Applying the Cauchy-Schwarz inequality together with the moment estimates  \eqref{eq moment bound},  \eqref{2f-2} and \eqref{eq Vij var}, we obtain 
       \begin{align*}
  & \mathbb E\left[   F^2(v(\tau_i, \lambda_j) )  F^2(v(\tau_{i'}, \lambda_{j'})) \left(   \Delta_{i,j}(V)^2   -\frac{1}{4N^2}\right)   \left(    \Delta_{i',j'}(V)^2   -\frac{1}{4N^2}\right) \right]\\
   =&\, \mathbb E\Bigg[   F^2(v(\tau_i, \lambda_j) )  F^2(v(\tau_{i'}, \lambda_{j'}))     \left(    \Delta_{i',j'}(V)^2   - \frac{1}{4N^2}\right) \\
   &   \ \ \ \ \ \  \ \ \ \ \     \cdot  \mathbb E\left[\Delta_{i,j}(V)^2   -\frac{1}{4N^2}\Big |\mathcal F_{\sqrt{2}(\tau_i+\lambda_j)}\right] \Bigg]\\
   \le &   \,    N^{-3}  \mathbb E\left[   F^2(v(\tau_i, \lambda_j) )  F^2(v(\tau_{i'}, \lambda_{j'}))    \left(    \Delta_{i',j'}(V)^2   -\frac{1}{4N^2}\right) \right]\\
  \le &\,  N^{-3} \left(\mathbb E\left[   F^8(v(\tau_i, \lambda_j) )  \right]\right)^{\frac14}  \cdot \left(\mathbb E\left[  F^8(v(\tau_{i'}, \lambda_{j'})) \right] \right)^{\frac14} \cdot \left(\mathbb E\left[ \left(\Delta_{i',j'}(V)^2   -\frac{1}{4N^2}\right)^2\right]\right)^{\frac12}    \\
  \lesssim  &\,   N^{-5}.
     \end{align*}

For pairs with   $(i, j), (i', j')\in \mathbb N^2$   satisfying $|(i+j)-( i'+j')|<2$,  we apply   the  Cauchy-Schwarz  inequality and the moment estimates   \eqref{eq moment bound} and \eqref{2f-2}  to get
           \begin{align*}
  & \mathbb E\left[   F^2(v(\tau_i, \lambda_j) )  F^2(v(\tau_{i'}, \lambda_{j'})) \left(   \Delta_{i,j}(V)^2   -\frac{1}{4N^2}\right)   \left(    \Delta_{i',j'}(V)^2   -\frac{1}{4N^2}\right) \right]\\
  \le &\, \left(\mathbb E\left[   F^8(v(\tau_i, \lambda_j) )  \right]\right)^{\frac14}  \cdot \left(\mathbb E\left[  F^8(v(\tau_{i'}, \lambda_{j'})) \right] \right)^{\frac14}\\
  &\, \,  \cdot  \left(\mathbb E\left[  \left(   \Delta_{i,j}(V)^2   -\frac{1}{4N^2}\right)^4\right]\right)^{\frac14}  \left(\mathbb E\left[  \left(   \Delta_{i',j'}(V)^2   -\frac{1}{4N^2}\right)^4\right]\right)^{\frac14}    \\
 \lesssim &\,   N^{-4}.
     \end{align*}
        
Summing over all   pairs, we obtain
 \begin{align*}
\mathbb E [|I_2|]\lesssim   N^{-1}.
  \end{align*}

  \noindent\textbf{Step 3:  Estimation of  $I_3$.} 
 By Proposition \ref{prop solution Holder},  the solution   $u$ is   H\"older continuous; consequently,     $v$ is also    H\"older continuous.
      Using the Lipschitz continuity of $F$, the Cauchy-Schwarz inequality,   and the moment estimates \eqref{eq moment bound} and \eqref{2f-2}, we  have
        \begin{align*}
  \mathbb E[|I_3|] \le &  \, \frac{1}{4} \sum_{i=0}^{N-1}\sum_{j=0}^{N-1} \int_{\tau_i}^{\tau_{i+1}}  \int_{\tau_j}^{\tau_{j+1}}    \mathbb E  \left[| F^2(v(\tau_i, \lambda_j)) - F^2(v(\tau, \lambda))|  \right]  d\tau d\lambda\\
   \lesssim  &  \,    \sum_{i=0}^{N-1}\sum_{j=0}^{N-1} \int_{\tau_i}^{\tau_{i+1}}  \int_{\tau_j}^{\tau_{j+1}}  \left(  \mathbb E  \left[|  v(\tau_i, \lambda_j) -v(\tau, \lambda)| ^2 \right] \right)^{\frac12} d\tau d\lambda\\
 \lesssim  &\,      N^{-\frac12}.
    \end{align*}  
    
    Combining the estimates for $I_1, I_2$, and  $I_3$, we obtain  \eqref{eq quad vc}.
         The proof is complete.  	      
 \end{proof}

  \subsection{Parameter estimates}   
In this section, we consider the problem of estimating the diffusion parameter in the  nonlinear damped stochastic Klein-Gordon equation:
       \begin{equation}\label{eq:DSKG para}
\begin{cases}
\vspace{6pt}
\displaystyle{\left (\DAlambert+a\partial_t+m^2 \right)u_{\theta}(t,x)= \theta F\left(u_{\theta}(t,x)\right)\dot    W(t,x), \quad t>0, x \in \mathbb R,}\\
\displaystyle{u_{\theta}(0, x) =0, \quad \frac{\partial}{\partial t} u_{\theta}(0, x) = 0.}
\end{cases}
\end{equation}  
 We estimate $\theta$ from discrete observations of $u_\theta$ on a space-time grid.
 
    For every $\tau>0$ and $\lambda\ge -\tau$,    let
       \begin{align}\label{eq v theta}
   v_{\theta}(\tau, \lambda):=  u_{\theta}\left(\frac{\tau+\lambda}{\sqrt{2}}, \frac{-\tau+\lambda}{\sqrt{2}}\right). 
\end{align}
     For every $N\ge1$, let 
       \begin{align}\label{eq V v}
Q_{N}(v_{\theta}):=  \sum_{i=0}^{N-1}\sum_{j=0}^{N-1}\Delta_{i,j}(v_{\theta})^2,
\end{align}
where 
$$
\Delta_{i,j}(v_{\theta})  := v_{\theta}\left(\tau_{i+1}, \lambda_{j+1}  \right)-v_{\theta}\left(\tau_{i+1}, \lambda_{j}\right)-v_{\theta}\left(\tau_{i}, \lambda_{j+1}  \right) +v_{\theta}\left(\tau_{i}, \lambda_{j}\right), 
$$ 
 and    $\tau_i:=\frac{i}{N}, \lambda_j:=\frac{j}{N}$. 
  
  Using Proposition~\ref{prop variation converge},  we know that
  \begin{align}\label{eq variation converge theta}
  Q_{N}(v_{\theta})\longrightarrow \frac{\theta^2 }{4}    \int_0^1\int_0^1     F^2 \big(v_{\theta}(\tau, \lambda) \big) d\tau d\lambda\, \,\, \text{in } L^{1}(\Omega), 
  \end{align}
 as $N\rightarrow\infty$.  
This convergence suggests  a consistent  estimator for the diffusion parameter $\theta$ based on the observations
$$\left\{u_{\theta}\left(\frac{\frac{i}{N}+\frac{j}{N} }{\sqrt 2}, \frac{-\frac{i}{N}+\frac{j}{N} }{\sqrt 2} \right);\, i, j = 0,1, \cdots, N\right\}.$$
  
In view of  the limit   \eqref{eq variation converge theta},  a natural way to estimate   $\theta$ is to use  the  approximation 
    \begin{align}\label{eq appr1}
   \theta  \approx \sqrt{\frac{ 4 Q_{N}(v_{\theta}) }{  \int_0^1\int_0^1   F^2\big(v_{\theta}(\tau, \lambda) \big) d\tau d\lambda}}.
    \end{align}
    To evaluate the  denominator,   we discretize the  double integral by  a   Riemann  sum:
    $$
    \int_0^1\int_0^1   F^2\big(v_{\theta}(\tau, \lambda) \big) d\tau d\lambda\approx N^{-2}  \sum_{i=0}^{N-1} \sum_{j=0}^{N-1}  F^2\big(v_{\theta}(\tau_i, \lambda_j) \big).
    $$
    This leads to the following estimator for   the diffusion parameter $\theta$:  
      \begin{align}\label{eq appr2}
    \hat \theta_N := \sqrt{\frac{4 N^2 Q_{N}(v_{\theta})}{   \sum_{i=0}^{N-1} \sum_{j=0}^{N-1}   F^2\big(v_{\theta}(\tau_i, \lambda_j) \big)}}, \ \ \ N\ge1.
    \end{align}

Based on Proposition \ref{prop variation converge}, one can directly obtain the following consistency result. For a detailed proof, see Corollary 4.1 in \cite{LW25}.  
  \begin{corollary}\label{coro est}   
 The estimator $  \hat \theta_N$ defined in  \eqref{eq appr2} is weakly consistent, i.e., 
$$ 
\hat \theta_N \stackrel{\mathbb P}{\longrightarrow}   \theta,  \ \ \ \ \text{as } N\rightarrow\infty.
 $$
\end{corollary}

 \vskip0.3cm 
 
\noindent{\bf Acknowledgments}  The research of G. Rang is partially supported by the NSF of  China (Nos. 12131019 and 11971361). The research of R. Wang is partially supported by the NSF of Hubei Province (No. 2024AFB683).

 \vskip0.3cm 
 
 \noindent{\bf Author Contributions}  G. Rang  and R. Wang designed the inference methodology, implemented the method, and drafted  the manuscript. All authors reviewed the manuscript.
\vskip0.3cm

\noindent{\bf {\large Declarations}}

\vskip0.3cm
\noindent{\bf Conflict of interest}  No potential conflict of interest was reported by the authors.

\vskip0.8cm

    \end{document}